\documentclass[12pt]{article}%
\usepackage{amsmath}
\usepackage{amsfonts}
\usepackage{amssymb}
\usepackage{graphicx}%
\providecommand{\U}[1]{\protect \rule{.1in}{.1in}}
\providecommand{\U}[1]{\protect \rule{.1in}{.1in}}
\newtheorem{theorem}{Theorem}

\newtheorem{remark}[theorem]{Remark}

\newenvironment{proof}[1][Proof]{\noindent \textbf{#1.} }{\hfill  \rule{0.5em}{0.5em}}
\begin{document}

\date{}
\title{Random matrix models of stochastic integral type for free infinitely divisible distributions}
\author{J. Armando Dom\'{\i}nguez Molina\thanks{jadguez@uas.uasnet.mx} \ and Alfonso
Rocha Arteaga\thanks{arteaga@uas.uasnet.mx}\\Facultad de Ciencias F\'{\i}sico-Matem\'{a}ticas\\Universidad Aut\'{o}noma de Sinaloa, M\'{e}xico}
\maketitle

\begin{abstract}
The Bercovici-Pata bijection maps the set of classical infinitely divisible
distributions to the set of free infinitely divisible distributions. The
purpose of this work is to study random matrix models for free infinitely
divisible distributions under this bijection. First, we find a specific form
of the polar decomposition for the L\'{e}vy measures of the random matrix
models considered in Benaych-Georges $\left[  6\right]  $ who introduced the
models through their measures. Second, random matrix models for free
infinitely divisible distributions are built consisting of infinitely
divisible matrix stochastic integrals whenever their corresponding classical
infinitely divisible distributions admit stochastic integral representations.
These random matrix models are realizations of random matrices given by
stochastic integrals with respect to matrix-valued L\'{e}vy processes.
Examples of these random matrix models for several classes of free infinitely
divisible distributions are given. In particular, it is shown that any free
selfdecomposable infinitely divisible distribution has a random matrix model
of Ornstein-Uhlenbeck type $\int_{0}^{\infty}e^{-t}d\Psi_{t}^{d}$, $d\geq1$,
where $\Psi_{t}^{d}$ is a $d\times d$ matrix-valued L\'{e}vy process
satisfying an $I_{\log}$ condition.

\end{abstract}

\section{Introduction}

An ensemble of random matrices is a sequence $\left(  M_{d}\right)  _{d\geq1}$
where $M_{d}$ is a $d\times d$ matrix whose entries are random variables. The
empirical spectral distribution of $M_{d}$ is the uniform distribution of its
spectrum $\lambda_{1},\lambda_{2},...,\lambda_{d}$, that is the (random)
probability measure $\mu_{M_{d}}$ defined as $\mu_{M_{d}}=\frac{1}{d}%
\sum_{i=1}^{d}\delta_{\lambda_{i}}$. A random matrix model for a probability
measure $\mu$ is an ensemble $\left(  M_{d}\right)  _{d\geq1}$ for which the
empirical spectral distribution $\mu_{M_{d}}$ converges weakly to $\mu$.

Bercovici and Pata \cite{BP} introduced a bijection $\Lambda$ from the set of
classical infinitely divisible distributions to the set of free infinitely
divisible distributions to study relations between classical and free
infinitely divisible aspects. Under this bijection Benaych-Georges \cite{BG}
and Cavanal-Duvillard \cite{CD} construct for any classical one-dimensional
infinitely divisible distribution $\mu$ a random matrix model for the
corresponding free infinitely divisible distribution $\Lambda \left(
\mu \right)  $. This include the Wigner and Marchenko-Pastur results, which
provide random matrix models of Gaussian Unitary Ensembles and Wishart random
matrices for the semicircle and Marchenko-Pastur distributions, respectively.

Specifically, it is shown in \cite{BG} that for any infinitely divisible
distribution $\mu$ on $\mathbb{R}$ there is a random matrix model of Hermitian
matrices $(M_{d})_{d\geq1}$ for $\Lambda(\mu)$. Moreover, for each $d\geq1$
the Fourier transform of $M_{d}$ is given by%
\begin{equation}
\mathbb{E}[\exp(i\mathrm{tr}(M_{d}A))]=\exp \left[  d\mathbb{E}_{u}%
\mathcal{C}_{\mu}\left(  \left \langle u,Au\right \rangle \right)  \right]
\text{ }A\text{ Hermitian,} \label{cumulantBG}%
\end{equation}
where $\mathcal{C}_{\mu}$ is the cumulant transform of $\mu$, $u$ is a
uniformly distributed column random vector in the unit sphere of
$\mathbb{C}^{d}$ and $\left \langle \cdot,\cdot \right \rangle $ is the usual
Hermitian product of $\mathbb{C}^{d}$.

Any $\mathbb{R}$-valued L\'{e}vy process $\left \{  X_{t}^{(\mu)}%
:t\geq0\right \}  $ with law $\mu$ at time $t=1$ has associated an $\mathbb{R}%
$-valued independently scattered random measure. The stochastic integral of a
real-valued function $h$ on $\left[  0,\infty \right)  $ with respect to
$X_{t}^{(\mu)}$ written as%
\begin{equation}
\int_{0}^{\infty}h(t)dX_{t}^{(\mu)}\text{,} \label{stochint}%
\end{equation}
is an $\mathbb{R}$-valued infinitely divisible random variable defined in the
sense of integrals of non random functions with respect to scattered random
measures, see Urbanik and Woyczynski \cite{UW} and Rajput\ and\ Rosi\'{n}ski
\cite{RR}; and Sato \cite{Sato2} for the $\mathbb{R}^{d}$ case. Several
important classes of infinitely divisible distributions having this stochastic
integral representation in law have been studied recently, see \cite{BNMS},
\cite{AM2}, \cite{AM}, \cite{AMR} and \cite{JRY}.

Let $I_{\log}(\mathbb{R})$ be the class infinitely divisible distributions
$\mu$ on $\mathbb{R}$ whose L\'{e}vy measures $\nu_{\mu}$ satisfy the
condition $\int_{\left \vert x\right \vert >2}\log \left \vert x\right \vert
\nu_{\mu}(dx)<\infty$. It is shown in Jurek and Vervaat \cite{JV}, Sato and
Yamazato \cite{SY} and Sato \cite{Sato1}, that the class of selfdecomposable
distributions on $\mathbb{R}$ is characterized by the stochastic integrals of
the form (\ref{stochint}) where $h(t)=e^{-t}$ and $\mu \in I_{\log}%
(\mathbb{R})$ in the following sense. For any $\mu \in I_{\log}(\mathbb{R})$
there exists a selfdecomposable distribution $\widetilde{\mu}$ such that%
\begin{equation}
\widetilde{\mu}=\mathcal{L}\left(  \int_{0}^{\infty}e^{-t}dX_{t}^{(\mu
)}\right)  \label{selfdecomprepresentation}%
\end{equation}
and vice versa, to any selfdecomposable distribution $\widetilde{\mu}$
corresponds a distribution $\mu$ in $I_{\log}(\mathbb{R})$ such that
(\ref{selfdecomprepresentation}) holds. This characterization of
selfdecomposable distributions as stochastic integrals is related to
Ornstein-Uhlenbeck type processes through the Langevin equation. The Langevin
equation $dY_{t}=dX_{t}^{(\mu)}-Y_{t}dt$ has stationary solution $\left \{
Y_{t}:t\geq0\right \}  $ if and only if $\mu \in I_{\log}(\mathbb{R})$. This
stationary solution $\left \{  Y_{t}\right \}  $ is unique and $\mathcal{L}%
\left(  Y_{t}\right)  =\mathcal{L}\left(  \int_{0}^{\infty}e^{-t}dX_{t}%
^{(\mu)}\right)  $ for all $t\geq0$. The process $\left \{  Y_{t}\right \}  $ is
called the stationary Ornstein-Uhlenbeck type process, see Sato \cite{Sato1}
and Rocha-Arteaga and Sato \cite{RAS}.

In this work we are concerned with random matrix models for free infinitely
divisible distributions corresponding to the image $\Lambda$ of classical
infinitely divisible distributions. We show that for every classical
one-dimensional infinitely divisible distribution representable as the
stochastic integral (\ref{stochint}) there exists a random matrix model for
the corresponding free infinitely divisible distribution, consisting of a
realization as matrix stochastic integral similar to (\ref{stochint}) with
respect to an appropriate matrix-valued L\'{e}vy process. In particular, the
free selfdecomposable distribution $\Lambda(\widetilde{\mu})$ corresponding to
the classical selfdecomposable distribution $\widetilde{\mu}$ with stochastic
integral representation (\ref{selfdecomprepresentation}) where $\mu \in
I_{\log}(\mathbb{R})$, has a realization as random matrix model of
Ornstein-Uhlenbeck type $\left(  \int_{0}^{\infty}e^{-t}d\Psi_{t}^{\mu
,d}\right)  _{d\geq1}$ where $\Psi_{t}^{\mu,d}$ is a $d\times d$ matrix-valued
L\'{e}vy process satisfying an $I_{\log}$-condition. Recently, P\'{e}rez-Abreu
and Sakuma \cite{PAS} studied random matrix models for free infinitely
divisible distributions as matrix stochastic integrals of Wiener-Gamma type.
They considered free infinitely divisible distributions corresponding to the
image $\Lambda$ of classical generalized Gamma convolutions distributions with
the so called Wiener-Gamma representation, a type of stochastic integral
representation (\ref{stochint}) with respect to the Gamma process.

This paper is organized as follows. In Section $2$ we find the polar
decomposition for the L\'{e}vy measures of the Hermitian matrices of the
random matrix models (\ref{cumulantBG}). In Section $3$ we construct random
matrix models for free infinitely divisible distributions as realizations of
classical matrix stochastic integrals with respect to matrix-valued L\'{e}vy
processes. We prove that these matrix stochastic integrals have the Fourier
transform (\ref{cumulantBG}). In Section $4$ we provide examples of random
matrix models of matrix stochastic integrals for several classes of free
infinitely divisible distributions. In particular, the class of free
selfdecomposable distributions corresponding to the image $\Lambda$ of the
class of selfdecomposable distributions in (\ref{selfdecomprepresentation}),
has random matrix models of Ornstein-Uhlenbeck type matrix integrals.

\section{Polar decomposition of L\'{e}vy measures of certain random matrix
models}

Let $\mathbb{M}_{d}=\mathbb{M}_{d\times d}\left(  \mathbb{C}\right)  $ denote
the linear space of $d\times d$ Hermitian matrices with scalar product
$\left \langle A,B\right \rangle =\ $\textrm{tr}$\left(  AB^{\ast}\right)  $ and
norm $\left \Vert A\right \Vert =\left[  \mathrm{tr}\left(  AA^{\ast}\right)
\right]  ^{1/2}$ where \textrm{tr} denotes trace. Let $\mathbb{\tilde{S}%
}_{\mathbb{M}_{d}}=\{V\in \mathbb{M}_{d}:\mathrm{rank}\left(  V\right)
=1,\allowbreak \left \Vert V\right \Vert =1\}$ the set of Hermitian matrices of
$\mathrm{rank}$ $1$ on the unit sphere of $\mathbb{M}_{d}$ and let
$\mathbb{\tilde{S}}_{\mathbb{M}_{d}}^{+}=\left \{  V>0:\text{ }\mathrm{rank}%
\left(  V\right)  =1,\left \Vert V\right \Vert =1\right \}  $ the set of positive
definite matrices of $\mathrm{rank}$ $1$ on the unit sphere of $\mathbb{M}%
_{d}$.

We recall the polar decomposition of L\'{e}vy measures on $\mathbb{R}$, see
\cite{R}\ and \cite{BNMS}. The L\'{e}vy measure $\nu$ of an infinitely
divisible distribution $\mu$ on $\mathbb{R}$ with $0<\nu \left(  \mathbb{R}%
\right)  \leq \infty$, can be expressed as%
\begin{equation}
\nu \left(  B\right)  =\int_{S}\lambda(d\xi)\int_{0}^{\infty}1_{B}(r\xi
)\nu_{\xi}(dr), \label{polardecomp}%
\end{equation}
where $\lambda$ is a measure on the unit sphere $S=\left \{  -1,1\right \}  $ of
$\mathbb{R}$ such that $0<\lambda \left(  S\right)  \leq \infty$ and $\nu_{\xi}$
is a measure on $(0,\infty)$ for each $\xi \in S$ such that $0<\nu_{\xi}\left(
(0,\infty)\right)  \leq \infty$. Here $\lambda$ and $\nu_{\xi}$ are called the
spherical and radial components of $\nu$, respectively.

Let $\widehat{\mu}$ and $\mathcal{C}_{\mu}$ denote the Fourier transform and
the cumulant transform of an infinitely divisible distribution $\mu$ on
$\mathbb{R}$, respectively. That is, $\mathcal{C}_{\mu}$ is the unique
continuous function from $\mathbb{R}$ into $\mathbb{C}$ such that
$\mathcal{C}_{\mu}\left(  0\right)  =0$ and $\widehat{\mu}(z)=\exp \left(
\mathcal{C}_{\mu}\left(  z\right)  \right)  $ for every $z\in \mathbb{R}$.

In \cite[Theorem 6.1]{BG} it is established that for any infinitely divisible
distribution $\mu$ on $\mathbb{R}$ there exists a random matrix model of
Hermitian matrices $(M_{d})_{d\geq1}$ for $\Lambda(\mu)$. Moreover, in
\cite[Theorem 3.1]{BG} the Fourier transform of $M_{d}$, for each $d\geq1$, is
given by%
\begin{equation}
\mathbb{E}[\exp(i\mathrm{tr}(M_{d}A))]=\exp \left[  d\mathbb{E}_{u}%
\mathcal{C}_{\mu}\left(  \left \langle u,Au\right \rangle \right)  \right]
\text{ }A\  \text{Hermitian} \label{FourierT BG}%
\end{equation}
where $\mathcal{C}_{\mu}$ is the cumulant transform of $\mu$, $u$ is a
uniformly distributed column random vector in the unit sphere of
$\mathbb{C}^{d}$ and $\left \langle \cdot,\cdot \right \rangle $ is the usual
Hermitian product of $\mathbb{C}^{d}$. In the sequel we denote by $\omega_{d}$
the probability measure on the set of matrices of \textrm{rank} $1$ induced by
the transformation
\begin{equation}
u\rightarrow V=uu^{\ast}\text{.} \label{inducedprobability}%
\end{equation}

In the following result we find the specific form of the polar decomposition
for the L\'{e}vy measures of the random matrix models considered in \cite{BG}.

\begin{theorem}
\label{polar}Let $\mu \ $be an infinitely divisible distribution on
$\mathbb{R}$ with L\'{e}vy measure $\nu$ and let $(M_{d})_{d\geq1}$ a random
matrix model for $\Lambda \left(  \mu \right)  $ where $M_{d}$ has the Fourier
transform (\ref{FourierT BG}). Then the L\'{e}vy measure $\nu_{M_{d}}$ of
$M_{d}$ is expressed as%
\[
\nu_{M_{d}}\left(  B\right)  =d\int_{\mathbb{\tilde{S}}_{\mathbb{M}_{d}}}%
\int_{0}^{\infty}1_{B}\left(  rV\right)  \nu_{V}\left(  dr\right)  \Pi \left(
dV\right)  \text{\quad}B\in \mathcal{B}\left(  \mathbb{M}_{d}\text{%
$\backslash$%
}\left \{  0\right \}  \right)  \text{,}%
\]
where $\nu_{V}=\nu^{+}$ or $\nu^{-}$ according to $V>0$ or\ $V<0$ and
$\Pi \left(  dV\right)  $ is a measure on $\mathbb{\tilde{S}}_{\mathbb{M}_{d}}$
such that
\[
\int \limits_{\mathbb{\tilde{S}}_{\mathbb{M}_{d}}}1_{D}\left(  V\right)
\Pi \left(  dV\right)  =\int \limits_{\mathbb{\tilde{S}}_{\mathbb{M}_{d}}^{+}%
}\int \limits_{\left \{  -1,1\right \}  }1_{D}\left(  \xi V\right)
\lambda \left(  d\xi \right)  \omega_{d}\left(  dV\right)  \text{\quad}%
D\in \mathcal{B}\left(  \mathbb{\tilde{S}}_{\mathbb{M}_{d}}\right)  \text{,}%
\]
where $\lambda$ is the spherical measure of $\nu$ and $\omega_{d}$ is the
probability measure on $\mathbb{\tilde{S}}_{\mathbb{M}_{d}}^{+}$ given by
(\ref{inducedprobability}).
\end{theorem}

\begin{proof}
Let $\lambda(d\xi)$ and $\nu_{\xi}$ be the spherical and radial components of
$\nu$ given by (\ref{polardecomp}), respectively. For every $z\in \mathbb{R}$,
$\mathcal{C}_{\mu}\left(  z\right)  =i\gamma z+\int_{\mathbb{R}}\left[
e^{izx}-1-ixz1_{\left \vert x\right \vert \leq1}\left(  x\right)  \right]
\nu \left(  dx\right)  $ where $\gamma \in \mathbb{R}$ and where we have omitted
the Gaussian term without loss of generality. From (\ref{FourierT BG}) we have
for every Hermitian matrix $A$,%
\begin{align*}
&  \log \mathbb{E}[\exp(i\mathrm{tr}(M_{d}A))]%
\genfrac{}{}{0pt}{}{\text{ }}{\text{ }}%
\\
&  =d\mathbb{E}_{u}\left \{  i\gamma \left \langle u,Au\right \rangle
+\int_{\mathbb{R}}\left[  e^{i\left \langle u,Au\right \rangle x}%
-1-ix\left \langle u,Au\right \rangle 1_{\left \vert x\right \vert \leq1}\left(
x\right)  \right]  \nu \left(  dx\right)  \right \}
\end{align*}%
\[
=i\gamma r_{A}d+d\int_{\mathbb{\tilde{S}}_{\mathbb{M}_{d}}^{+}}\int
_{\mathbb{R}}\left[  e^{i\mathrm{tr}\left(  Auu^{\ast}\right)  x}%
-1-ix\mathrm{tr}\left(  Auu^{\ast}\right)  1_{\left \vert x\right \vert \leq
1}\left(  x\right)  \right]  \nu \left(  dx\right)  \omega_{d}\left(
dV\right)
\]%
\[
=i\gamma r_{A}d+d\int \nolimits_{\mathbb{\tilde{S}}_{\mathbb{M}_{d}}^{+}}%
\int_{\left \{  -1,1\right \}  }\int_{0}^{\infty}\left[  e^{i\mathrm{tr}\left(
AV\right)  r\xi}-1-ir\xi \mathrm{tr}\left(  AV\right)  1_{\left \vert
r\xi \right \vert \leq1}\left(  r\xi \right)  \right]  \nu_{\xi}\left(
dr\right)  \lambda \left(  d\xi \right)  \omega_{d}\left(  dV\right)
\]%
\[
=i\gamma r_{A}d+d\int_{\mathbb{\tilde{S}}_{\mathbb{M}_{d}}}\int_{0}^{\infty
}e^{i\mathrm{tr}\left(  A\tilde{V}\right)  r}-1-ir\mathrm{tr}\left(
A\tilde{V}\right)  1_{r\leq1}\left(  r\right)  \nu_{\xi}\left(  dr\right)
\Pi \left(  d\tilde{V}\right)  ,
\]
where $r_{A}=\mathbb{E}_{u}\left \langle u,Au\right \rangle $ and $\Pi \left(
d\tilde{V}\right)  $ is a measure on $\mathbb{\tilde{S}}_{\mathbb{M}_{d}}$
such that for any Borel set $D$ of $\mathbb{\tilde{S}}_{\mathbb{M}_{d}}$%
\[
\int \limits_{\left \{  \tilde{V}:\text{ }\mathrm{rank}\left(  \tilde{V}\right)
=1,\left \Vert \tilde{V}\right \Vert =1\right \}  }1_{D}\left(  \tilde{V}\right)
\Pi \left(  d\tilde{V}\right)  =\int \limits_{\left \{  V>0:\text{ }%
\mathrm{rank}\left(  V\right)  =1,\left \Vert V\right \Vert =1\right \}  }%
\int \limits_{\left \{  -1,1\right \}  }1_{D}\left(  \xi V\right)  \lambda \left(
d\xi \right)  \omega_{d}\left(  dV\right)  \text{,}%
\]
here the equivalence of the regions of integration follows from the spectral
representation theorem, since any $V$ with $\mathrm{rank}\left(  V\right)  =1$
and $\left \Vert V\right \Vert =1$ can be written as $V=\lambda vv^{\ast}$ where
$\lambda$ is eigenvalue of $V$ with corresponding eigenvector $v$, hence
$x^{\ast}Vx=\lambda \left \vert x^{\ast}v\right \vert ^{2}$ and therefore $V>0$
or $V<0$ according to the sign of $\lambda$.
\end{proof}

\section{Random matrix models of stochastic integral type}

Any real-valued L\'{e}vy process $\left \{  X_{t}^{(\mu)}:t\geq0\right \}  $
with law $\mu$ at time $t=1$, uniquely induces a real-valued independently
scattered random measure $\left \{  M^{(\mu)}\left(  B\right)  :B\in
\mathcal{B}_{0}\left(  \left[  0,\infty \right)  \right)  \right \}  $ such that
$M^{(\mu)}\left(  \left[  0,t\right]  \right)  =X_{t}^{(\mu)}$ almost surely,
where $\mathcal{B}_{0}\left(  \left[  0,\infty \right)  \right)  $ is the
family of bounded Borel sets in $\left[  0,\infty \right)  $. We will consider
$M^{(\mu)}$-integrable $\left(  \text{or }X_{t}^{(\mu)}\text{-integrable}%
\right)  $ real-valued functions $h$ on $\left[  0,\infty \right)  $ in the
sense of \cite{RR} and \cite{Sato2}. Then $\int_{B}h(t)M^{(\mu)}(dt)$ $\left(
\text{or }\int_{B}h(t)dX_{t}^{(\mu)})\right)  $ is defined almost surely for
every $B\in \mathcal{B}_{0}\left(  \left[  0,\infty \right)  \right)  $. The
stochastic integral of a real-valued $h$ on $\left[  0,\infty \right)  $ with
respect to $X_{t}^{(\mu)}$ is a real-valued infinitely divisible random
variable written as%
\begin{equation}
\eta=\int_{0}^{\infty}h(t)dX_{t}^{(\mu)}\text{,} \label{stointeg}%
\end{equation}
which is defined as the limit in probability of $\int_{\left[  0,s\right]
}h(t)dX_{t}^{(\mu)}$ as $s\rightarrow \infty$ whenever the limit exists.
Furthermore, its cumulant transform is given by%
\begin{equation}
\mathcal{C}_{\eta}\left(  z\right)  =\int_{0}^{\infty}\mathcal{C}_{\mu}\left(
h(t)z\right)  dt\quad z\in \mathbb{R}\text{.} \label{ctr}%
\end{equation}
\bigskip

For the complex matrix case we have a similar result; see \cite{BNS} for the
case of real matrices. For any infinitely divisible matrix $\Psi$ in
$\mathbb{M}_{d}$ with associated matrix L\'{e}vy process $\left \{  \Psi
_{t}^{d}:t\geq0\right \}  $, the infinitely divisible $d\times d$ matrix valued
stochastic integral
\begin{equation}
M=\int_{0}^{\infty}h(t)d\Psi_{t}^{d}\text{,} \label{sim}%
\end{equation}
whenever exists, has cumulant transform%
\begin{equation}
\mathcal{C}_{M}\left(  A\right)  =\int_{0}^{\infty}\mathcal{C}_{\Psi}\left(
h(t)A\right)  dt\quad A\in \mathbb{M}_{d}\text{.} \label{ctm}%
\end{equation}
In this work we consider matrix L\'{e}vy processes $\left \{  \Psi_{t}%
^{d}\right \}  $ corresponding to L\'{e}vy measures of the form
\begin{equation}
\nu_{\Psi}^{d}(B)=d\int_{\mathbb{\tilde{S}}_{\mathbb{M}_{d}}^{+}}\omega
_{d}\left(  dV\right)  \int_{\mathbb{R}}1_{B}\left(  xV\right)  \nu_{\mu
}\left(  dx\right)  \text{,} \label{levymeasure}%
\end{equation}
where $\omega_{d}$ is the probability measure induced by the transformation
$u\rightarrow V=uu^{\ast}$ in (\ref{inducedprobability}) and $\nu_{\mu}$ is a
L\'{e}vy measure of an infinitely divisible distribution $\mu$ on $\mathbb{R}%
$. Observe that $\nu_{\Psi}^{d}$ is L\'{e}vy measure supported in the subset
of $\mathrm{rank}$ one matrices in $\mathbb{M}_{d}$.

The following is the main result of this work, it provides random matrix
models for free infinitely divisible distributions on $\mathbb{R}$ given by
matrix stochastic integrals of the form (\ref{sim}) when the corresponding
classical infinitely divisible distributions under $\Lambda$ are representable
as the random integrals of the form (\ref{stointeg}).

\begin{theorem}
\label{matrixintegral}Let $\mu_{h}$ be an infinitely divisible distribution on
$\mathbb{R}$ given by the stochastic integral representation%
\begin{equation}
\mu_{h}=\mathcal{L}\left(  \int_{0}^{\infty}h(t)dX_{t}^{\left(  \mu \right)
}\right)  \text{,} \label{ir}%
\end{equation}
where $X_{t}^{\left(  \mu \right)  }$ is a L\'{e}vy process on $\mathbb{R}$
with law $\mu$ at time $t=1$ and L\'{e}vy measure $\nu_{\mu}$. The free
infinitely divisible distribution $\Lambda \left(  \mu_{h}\right)  $ has a
random matrix model given by the ensemble of infinitely divisible matrix
stochastic integrals%
\begin{equation}
\left(  M_{h}^{d}=\int_{0}^{\infty}h(t)d\Psi_{t}^{d}\right)  _{d\geq1}\text{,}
\label{rmm}%
\end{equation}
where $\Psi_{t}^{d}$ is the $\mathbb{M}_{d}$-valued L\'{e}vy process with
L\'{e}vy measure $\nu_{\Psi}^{d}$ given by (\ref{levymeasure}) in terms of
$\omega_{d}$ and $\nu_{\mu}$.
\end{theorem}

\begin{proof}
We will proof that the random matrices of the ensemble $\left(  M_{h}%
^{d}\right)  _{d\geq1}$ in (\ref{rmm}) and the random matrices of the random
matrix model $(M_{d})_{d\geq1}$ for $\Lambda \left(  \mu_{h}\right)  $ given in
\cite[Theorem 6.1]{BG} have the same laws. Recall from (\ref{FourierT BG})
that the Fourier transform of $M_{d}$ is given by
\[
\mathbb{E}[\exp(i\mathrm{tr}(M_{d}A))]=\exp \left[  d\mathbb{E}_{u}%
\mathcal{C}_{\mu_{h}}\left(  \left \langle u,Au\right \rangle \right)  \right]
\ A\text{ Hermitian.}%
\]
We will prove that this Fourier transform of $M_{d}$ coincides with the
Fourier transform of $M_{h}^{d}$. For that, we first calculate the cumulant
transform of $\Psi_{t}^{d}$ at time $1$ using (\ref{levymeasure}),%
\begin{align}
\mathcal{C}_{\Psi_{1}^{d}}\left(  A\right)   &  =\int_{\mathbb{M}_{d}}\left[
e^{i\mathrm{tr}(AX)}-1-i\mathrm{tr}(AX)1_{\left \Vert X\right \Vert \leq
1}\left(  X\right)  \right]  \nu_{\Psi_{1}}^{d}\left(  dX\right) \nonumber \\
&  =d\int_{\mathbb{\tilde{S}}_{\mathbb{M}_{d}}^{+}}\int_{\mathbb{R}}\left[
e^{i\mathrm{tr}(AV)x}-1-ix\mathrm{tr}(AV)1_{\left \vert x\right \vert \leq
1}\left(  x\right)  \right]  \omega_{d}\left(  dV\right)  \nu_{\mu}\left(
dx\right) \nonumber \\
&  =d\int_{\mathbb{R}}\mathbb{E}_{V}\left[  e^{i\mathrm{tr}(AV)x}%
-1-ix\mathrm{tr}(AV)1_{\left \vert x\right \vert \leq1}\left(  x\right)
\right]  \nu_{\mu}\left(  dx\right) \nonumber \\
&  =d\mathbb{E}_{u}\int_{\mathbb{R}}\left[  e^{i\mathrm{tr}(Auu^{\ast}%
)x}-1-ix\mathrm{tr}(Auu^{\ast})1_{\left \vert x\right \vert \leq1}\left(
x\right)  \right]  \nu_{\mu}\left(  dx\right) \nonumber \\
&  =d\mathbb{E}_{u}\mathcal{C}_{\mu}\left(  \left \langle u,Au\right \rangle
\right)  \text{.} \label{cumulant}%
\end{align}
Now we calculate the cumulant transform of $M_{h}^{d}$ from (\ref{ctm}) and
(\ref{cumulant}),%
\begin{align*}
\mathcal{C}_{M_{h}^{d}}\left(  A\right)   &  =\int_{0}^{\infty}\mathcal{C}%
_{\Psi_{1}^{d}}\left(  h\left(  t\right)  A\right)  dt=\int_{0}^{\infty
}d\mathbb{E}_{u}\mathcal{C}_{\mu}\left(  \left \langle u,h\left(  t\right)
Au\right \rangle \right)  dt\\
&  =d\mathbb{E}_{u}\int_{0}^{\infty}\mathcal{C}_{\mu}\left(  h\left(
t\right)  \left \langle u,Au\right \rangle \right)  dt=d\mathbb{E}%
_{u}\mathcal{C}_{\mu_{h}}\left(  \left \langle u,Au\right \rangle \right)
\text{,}%
\end{align*}
where in the last equality we have used the relation (\ref{ctr}) between the
cumulant transforms of $\mu_{h}$ and $\mu$ corresponding to the stochastic
integral representation (\ref{ir}).
\end{proof}

\begin{remark}
Let $\nu_{\mu},\nu_{\mu_{h}},\nu_{M_{h}^{d}}$ and $\nu_{\Psi^{d}}$ denote the
L\'{e}vy measures of $\mu,\mu_{h},M_{h}^{d}\ $and $\Psi_{t}^{d} $ at time $1$
in Theorem \ref{matrixintegral}, respectively. From \cite{Sato2} the L\'{e}vy
measures of $\mu$ and $\mu_{h}$ are related as follows%
\[
\nu_{\mu_{h}}\left(  B\right)  =\int_{0}^{\infty}dt\int_{\mathbb{R}}%
1_{B}(h(t)x)\nu_{\mu}\left(  dx\right)  \text{\quad}B\in \mathcal{B}\left(
\mathbb{R}\text{%
$\backslash$%
}\left \{  0\right \}  \right)
\]
and from \cite{BNS} it is obtained a similar relation between the L\'{e}vy
measures of $M_{h}^{d}\ $and $\Psi^{d}$,%
\[
\nu_{M_{h}^{d}}\left(  B\right)  =\int_{0}^{\infty}dt\int_{\mathbb{M}_{d}%
}1_{B}(h(t)X)\nu_{\Psi^{d}}\left(  dX\right)  \text{\quad}B\in \mathcal{B}%
\left(  \mathbb{M}_{d}\text{%
$\backslash$%
}\left \{  0\right \}  \right)  \text{.}%
\]
These two relations combined with (\ref{levymeasure}) yield the L\'{e}vy
measure of $M_{h}^{d}$ in terms of the L\'{e}vy measure of $\mu_{h}$%
\begin{align}
\nu_{M_{h}^{d}}\left(  B\right)   &  =d\int_{\mathbb{\tilde{S}}_{\mathbb{M}%
_{d}}^{+}}\omega_{d}\left(  dV\right)  \int_{\mathbb{R}}\nu_{\mu}\left(
dx\right)  \int_{0}^{\infty}1_{B}(h(t)xV)dt\nonumber \\
&  =d\int_{\mathbb{\tilde{S}}_{\mathbb{M}_{d}}^{+}}\omega_{d}\left(
dV\right)  \int_{\mathbb{R}}\nu_{\mu_{h}}\left(  dx\right)  1_{B}%
(xV)\text{,\quad}B\in \mathcal{B}\left(  \mathbb{M}_{d}\text{%
$\backslash$%
}\left \{  0\right \}  \right)  \text{.} \label{Levymasurermm}%
\end{align}

\end{remark}

\section{Examples}

In the following $I(\mathbb{R})$ denotes the class of infinitely divisible
distributions on $\mathbb{R}$ and $L(\mathbb{R})$ denotes the class of
selfdecomposable distributions on $\mathbb{R}$.

Let $\left(  \lambda,\nu_{\xi}\right)  $ denote the polar decomposition of the
L\'{e}vy measure $\nu$ of any $\mu \in I(\mathbb{R})$ given by
(\ref{polardecomp}), that is
\[
\nu \left(  B\right)  =\int_{S}\lambda(d\xi)\int_{0}^{\infty}1_{B}(r\xi
)\nu_{\xi}(dr)\text{,}%
\]
where $\lambda$ and $\nu_{\xi}$ are the spherical and radial components of
$\nu$, respectively and $S=\left \{  -1,1\right \}  $. Moreover, let $\left(
\lambda,k_{\xi}\right)  $ denote the following description of the above
descomposition for selfdecomposable distributions, see \cite{Sato1}. If $\mu$
$\in L(\mathbb{R})\,$the radial component of its L\'{e}vy measure $\nu$ is
expressed as
\[
\nu_{\xi}(dr)=1_{(0,\infty)}(r)\frac{k_{\xi}(r)}{r}dr\text{,}%
\]
where $k_{\xi}(r)$ is a nonnegative measurable function in $\xi \in S$ and
decreasing, right continuous in $r\in(0,\infty)$. Here $k_{\xi}$ is called the
$k$-function of $\nu$.

Throughout this section we provide examples of random matrix models for free
infinitely divisible distributions $\Lambda \left(  \mu_{h}\right)  $ given by
the ensembles of matrix stochastic integrals (\ref{rmm}), for several classes
infinitely divisible distributions $\mu_{h}$ with stochastic integral
representation (\ref{ir}). In most of examples we express the L\'{e}vy measure
$\nu_{M_{h}^{d}}$ in terms of $\nu_{\mu_{h}}$. We emphasize, according to
(\ref{Levymasurermm}), that these L\'{e}vy measures $\nu_{M_{h}^{d}}$ are
supported in the subset of $\mathrm{rank}$ one matrices in $\mathbb{M}_{d}$.

In all examples below $\nu_{\mu},\nu_{\mu_{h}},\nu_{M_{h}^{d}}$ and $\nu
_{\Psi^{d}}$ denote the L\'{e}vy measures of $\mu,\mu_{h},M_{h}^{d}\ $and
$\Psi_{t}^{d}$ at time $1$ in Theorem \ref{matrixintegral}, respectively.

The following two examples recover the random matrix models of Theorems $4.1$
and $4.3$ in \cite{PAS} and find their corresponding L\'{e}vy measures.

\textbf{Example 1}. The class of Generalized Gamma Convolutions $T(\mathbb{R}%
_{+})$ is the smallest class of infinitely divisible distributions on
$\mathbb{R}_{+}$ that contains all Gamma distributions and that is closed
under convolution and weak convergence.

In \cite{JRY}, any distribution in $T(\mathbb{R}_{+})$ has stochastic integral
representation (\ref{ir}) where $\mu$ is a Gamma distribution and $h(t)$ is a
Borel measurable function $h:\mathbb{R}_{+}\rightarrow \mathbb{R}_{+}$ such
that $\int_{0}^{\infty}\log \left(  1+h(t)\right)  dt<\infty$. Such a
representation is called the Wiener-Gamma integral representation.

Let $\mu_{h}\in T(\mathbb{R}_{+})$ has the Wiener-Gamma integral
representation. By $(8)$ in \cite{PAS} the L\'{e}vy measure of $\mu_{h}$ is
expressed as%
\[
\nu_{\mu_{h}}\left(  dx\right)  =1_{\left(  0,\infty \right)  }(x)\int
_{0}^{\infty}\frac{e^{-x/h\left(  t\right)  }}{x}dtdx\text{.}%
\]
The corresponding random matrix models (\ref{rmm}) for the free Generalized
Gamma Convolutions $\Lambda(\mu_{h})$ consist of matrix stochastic integrals
of Wiener-Gamma type%
\[
\left(  M_{h}^{d}=\int_{0}^{\infty}h(t)d\Psi_{t}^{d}\right)  _{d\geq1}%
\]
with L\'{e}vy measures%
\[
\nu_{M_{h}^{d}}(B)=d\int_{\mathbb{\tilde{S}}_{\mathbb{M}_{d}}^{+}}\omega
_{d}\left(  dV\right)  \int_{0}^{\infty}dx1_{B}(xV)\int_{0}^{\infty}%
\frac{e^{-x/h\left(  t\right)  }}{x}dt.
\]

\textbf{Example 2}. The class of Thorin distributions $T(\mathbb{R})$ is the
smallest class of infinitely divisible distributions on $\mathbb{R}$ which
contains all distributions in $T(\mathbb{R}_{+})$ and is closed under
convolution, weak convergence and reflection.

It is shown in \cite{BNMS} that this class is characterized by the stochastic
integral representation (\ref{ir}) where $\mu \in ID_{\log}(\mathbb{R})$ and
$h(t)$ is the inverse function of the incomplete Gamma function $g(t)=\int
_{t}^{\infty}e^{-s}s^{-1}ds$.

Let $\mu_{h}\in T(\mathbb{R})$ with such a representation. The random matrix
models (\ref{rmm}) for the corresponding free Thorin distributions
$\Lambda(\mu_{h})$ are given by%
\[
\left(  M_{h}^{d}=\int_{0}^{\infty}g^{\ast}(t)d\Psi_{t}^{d}\right)  _{d\geq
1}\text{\quad}g^{\ast}\text{ inverse function of }g\text{.}%
\]

Let $(\lambda,\nu_{\xi})$ and $(\lambda_{h},\nu_{h_{\xi}})$ be the polar
decompositions of the L\'{e}vy measures of $\mu$ and $\mu_{h}$, respectively.
It is proved in \cite{AM} that the corresponding $k$-function of $\mu_{h}$ is
$k_{h_{\xi}}(r)=\int_{0}^{\infty}e^{-r/s}\nu_{\xi}(ds)$. Therefore the
L\'{e}vy measures of these matrix stochastic integrals are given by%
\[
\nu_{M_{h}^{d}}(B)=d\int_{\mathbb{\tilde{S}}_{\mathbb{M}_{d}}^{+}}\omega
_{d}\left(  dV\right)  \int_{S}\lambda_{h}(d\xi)\int_{0}^{\infty}\frac{dr}%
{r}\left(  \int_{0}^{\infty}e^{-r/s}\nu_{\xi}(ds)\right)  1_{B}(r\xi V).
\]

\textbf{Example 3.} The class of Bondesson distributions $B(\mathbb{R})$ is
the smallest class of infinitely divisible distributions on $\mathbb{R}$ that
contains all mixtures of exponential distributions and that is closed under
convolution, weak convergence and reflection.

This class is characterized by the stochastic integral representation
(\ref{ir}) where $\mu \in I(\mathbb{R})$ and $h(t)=1_{\left(  0,1\right)
}(t)\log \left(  1/t\right)  $, see \cite{BNMS}. If $\mu_{h}\in B(\mathbb{R})$
has such a representation then by $(2.17)$ in \cite{BNMS}%
\[
\nu_{\mu_{h}}\left(  dx\right)  =\int_{0}^{\infty}e^{-s}\nu_{\mu}\left(
s^{-1}dx\right)  ds\text{.}%
\]
Therefore the random matrix models (\ref{rmm}) for the free Bondesson
distributions $\Lambda(\mu_{h})$ are given by%
\[
\left(  M_{h}^{d}=\int_{0}^{1}\log \left(  1/t\right)  d\Psi_{t}^{d}\right)
_{d\geq1}%
\]
with L\'{e}vy measures%
\[
\nu_{M_{h}^{d}}(B)=d\int_{\mathbb{\tilde{S}}_{\mathbb{M}_{d}}^{+}}\omega
_{d}\left(  dV\right)  \int_{-\infty}^{\infty}\int_{0}^{\infty}e^{-s}\nu_{\mu
}\left(  s^{-1}dx\right)  ds1_{B}(xV)\text{.}%
\]

\textbf{Example 4}. The class of Thorin distributions $T(\mathbb{R})$ is also
characterized (see Example $2$) by the stochastic integral representation
(\ref{ir}) where $\mu \in L\left(  \mathbb{R}\right)  $ and $h(t)=1_{\left(
0,1\right)  }(t)\log \left(  1/t\right)  $, see \cite{BNMS}.\newline If
$\mu_{h}\in T(\mathbb{R})$ has this alternative representation, the
corresponding random matrix models (\ref{rmm}) for the free Thorin
distributions $\Lambda(\mu_{h})$ are%
\[
\left(  M_{h}^{d}=\int_{0}^{1}\log \left(  1/t\right)  d\Psi_{t}^{d}\right)
_{d\geq1}\text{.}%
\]

Let $(\lambda,k_{\xi})$ and $(\lambda_{h},k_{h_{\xi}})$ be the corresponding
polar decompositions of the L\'{e}vy measures of $\mu$ and $\mu_{h}$,
respectively. It is shown in \cite{BNMS} that $\lambda=\lambda_{h}$ and%
\[
k_{h_{\xi}}(r)=\int_{0}^{\infty}k_{\xi}(rs^{-1})e^{-s}ds
\]
and hence the L\'{e}vy measures for these random matrix models are of the form%
\begin{align*}
\nu_{M_{h}^{d}}(B)  &  =d\int_{\mathbb{\tilde{S}}_{\mathbb{M}_{d}}^{+}}%
\omega_{d}\left(  dV\right)  \int_{-\infty}^{\infty}1_{B}(xV)\nu_{\mu_{h}%
}\left(  dx\right) \\
&  =d\int_{\mathbb{\tilde{S}}_{\mathbb{M}_{d}}^{+}}\omega_{d}\left(
dV\right)  \int_{S}\lambda \left(  d\xi \right)  \int_{0}^{\infty}\frac{dr}%
{r}\left(  \int_{0}^{\infty}k_{\xi}(rs^{-1})e^{-s}ds\right)  1_{B}(r\xi V).
\end{align*}

\textbf{Example 5.} The class of selfdecomposable distributions $L\left(
\mathbb{R}\right)  $ is characterized by the stochastic integral
representation (\ref{ir}) where $\mu \in I_{\log}(\mathbb{R})$ and
$h(t)=1_{\left(  0,\infty \right)  }(t)e^{-t}$, see \cite{JV}, \cite{SY} and
\cite{Sato1}.

Let $\mu_{h}\in L\left(  \mathbb{R}\right)  $ with such a Ornstein-Uhlenbeck
type integral representation. The random matrix models (\ref{rmm}) for the
corresponding free selfdecomposable distributions $\Lambda(\mu_{h})$ are given
by%
\[
\left(  M_{h}^{d}=\int_{0}^{\infty}e^{-t}d\Psi_{t}^{d}\right)  _{d\geq1}%
\]
and satisfy an $I_{\log}$-condition, that is%
\[
\int_{\left \Vert X\right \Vert >2}\log \left \Vert X\right \Vert \nu_{\Psi}%
^{d}(dX)<\infty \text{,}%
\]
which follows easily from (\ref{levymeasure}) and the fact that $\mu \in
I_{\log}(\mathbb{R})$. These Ornstein-Uhlenbeck type matrix integrals are
supported on the subset of \textrm{rank} one matrices in $\mathbb{M}_{d}%
$.\newline Let $(\lambda,\nu_{\xi})$ and $(\lambda_{h},k_{h_{\xi}})$ be the
corresponding polar decompositions of the L\'{e}vy measures of $\mu$ and
$\mu_{h}$. It is proved in \cite{AM} that the $k$-function of $\nu_{\mu_{h}}$
is given by%
\[
k_{h_{\xi}}(r)=\nu_{\xi}\left(  \left(  r,\infty \right)  \right)  \text{.}%
\]
Therefore, the L\'{e}vy measures of these random matrix models of
Ornstein-Uhlenbeck type for free selfdecomposable distributions $\Lambda
(\mu_{h})$ are given by%
\begin{align*}
\nu_{M_{h}^{d}}(B)  &  =d\int_{\mathbb{\tilde{S}}_{\mathbb{M}_{d}}^{+}}%
\omega_{d}\left(  dV\right)  \int_{-\infty}^{\infty}1_{B}(xV)\nu_{\mu_{h}%
}\left(  dx\right) \\
&  =d\int_{\mathbb{\tilde{S}}_{\mathbb{M}_{d}}^{+}}\omega_{d}\left(
dV\right)  \int_{S}\lambda_{h}\left(  d\xi \right)  \int_{0}^{\infty}\frac
{dr}{r}\nu_{\xi}\left(  \left(  r,\infty \right)  \right)  1_{B}(r\xi V).
\end{align*}

\textbf{Example 6.} The class$\mathfrak{\ }$of type $G$ distributions
$G_{sym}(\mathbb{R})$ is the class of symmetric distributions on $\mathbb{R}$
which are variance mixtures of the standard Gaussian distribution with
positive infinitely divisible mixing distributions.

In \cite{AM2} is proved that this class is characterized by the stochastic
integral representation (\ref{ir}) where $\mu \in I(\mathbb{R})$ and $h(t)$ is
the inverse function of $f(t)=\int_{t}^{\infty}\varphi(u)du$ where
$\varphi(u)=\frac{1}{\sqrt{2\pi}}e^{-\frac{u^{2}}{2}}$.

Let $\mu_{h}\in G_{sym}(\mathbb{R})$ with such a representation. In this case
the random matrix models (\ref{rmm}) for the corresponding free type $G$
distributions $\Lambda(\mu_{h})$ are of the form%
\[
\left(  M_{h}^{d}=\int_{0}^{\infty}f^{\ast}(t)d\Psi_{t}^{d}\right)  _{d\geq
1}\quad f^{\ast}\text{ inverse function of }f\text{.}%
\]
\newline Let $(\lambda_{h},\nu_{h_{\xi}})$ be the polar decomposition of the
L\'{e}vy measure of the type $G$ distribution $\mu_{h}$. It is known that the
radial component $\nu_{h_{\xi}}$ can be written as
\[
\nu_{h_{\xi}}\left(  dr\right)  =g_{h_{\xi}}(r^{2})dr
\]
where $g_{h_{\xi}}(r)$ is measurable in $\xi \in S$ and completely monotone in
$r\in(0,\infty)$. It is proved in \cite{AM} that
\[
g_{h_{\xi}}(r)=\int_{0}^{\infty}\varphi(r^{1/2}/s)s^{-1}\nu_{\xi}(ds)\text{,}%
\]
where $\nu_{\xi}$ is the radial component of $\nu_{\mu}$. Therefore the
L\'{e}vy measures for these random matrix models for free type $G$
distributions $\Lambda(\mu_{h})$ are given by%
\[
\nu_{M_{h}^{d}}(B)=d\int_{\mathbb{\tilde{S}}_{\mathbb{M}_{d}}^{+}}\omega
_{d}\left(  dV\right)  \int_{S}\lambda_{h}\left(  d\xi \right)  \int
_{0}^{\infty}dr\left(  \int_{0}^{\infty}\varphi(r/s)s^{-1}\nu_{\xi
}(ds)\right)  1_{B}(r\xi V)\text{.}%
\]

\textbf{Example 7}. The class$\mathfrak{\ }$of $M$ distributions
$M(\mathbb{R})$ is a subclass of selfdecomposable distributions of
$G_{sym}(\mathbb{R})$. It is the class of symmetric infinitely divisible
distributions on $\mathbb{R}$ whose L\'{e}vy measures have polar decomposition
$(\lambda,\nu_{\xi})$ such that%
\begin{equation}
\nu_{\xi}\left(  dr\right)  =\frac{g_{\xi}(r^{2})}{r}dr\text{,}
\label{gfunction}%
\end{equation}
where $g_{h_{\xi}}(r)$ is measurable in $\xi \in S$ and completely monotone in
$r\in(0,\infty)$.\bigskip

In \cite{AMR} is shown that the class$\mathfrak{\ }M(\mathbb{R})$ is
characterized by the stochastic integral representation (\ref{ir}) where
$\mu \in I_{\log}(\mathbb{R})$ and $h(t)$ is the inverse function of
$m(t)=\int_{t}^{\infty}\frac{\varphi(u)}{u}du$.

Let $\mu_{h}\in M(\mathbb{R})$ with such a representation. The random matrix
models (\ref{rmm}) for the corresponding free $M$ distributions $\Lambda
(\mu_{h})$ are%
\[
\left(  M_{h}^{d}=\int_{0}^{\infty}m^{\ast}(t)d\Psi_{t}^{d}\right)  _{d\geq
1}\quad m^{\ast}\text{ inverse function of }m\text{.}%
\]
Let $(\lambda_{h},\nu_{h_{\xi}})$ be the polar decomposition of the L\'{e}vy
measure of $\mu_{h}$ and let $g_{h_{\xi}}$ be the corresponding $g$-function
in (\ref{gfunction}) of the radial component $\nu_{h_{\xi}}$. It is proved in
\cite{AM} that
\[
g_{h_{\xi}}(r)=\int_{0}^{\infty}\varphi(r^{1/2}/s)\nu_{\xi}(ds)\text{,}%
\]
where $\nu_{\xi}$ is the radial component of $\nu_{\mu}$. Therefore the
L\'{e}vy measures for these random matrix models for the free $M$
distributions $\Lambda(\mu_{h})$ can written as%
\[
\nu_{M_{h}^{d}}(B)=d\int_{\mathbb{\tilde{S}}_{\mathbb{M}_{d}}^{+}}\omega
_{d}\left(  dV\right)  \int_{S}\lambda_{h}\left(  d\xi \right)  \int
_{0}^{\infty}dr\left(  \int_{0}^{\infty}\varphi(r/s)\nu_{\xi}(ds)\right)
1_{B}(r\xi V)\text{.}%
\]

\textbf{Example 8.} The class of Jurek distributions $U(\mathbb{R})$ is the
class of infinitely divisible distributions on $\mathbb{R}$ whose L\'{e}vy
measures have polar decomposition $(\lambda,\nu_{\xi})$ such that
\begin{equation}
\nu_{\xi}\left(  dr\right)  =l_{\xi}(r)dr\text{,} \label{lfunction}%
\end{equation}
where $l_{\xi}$ is measurable in $\xi \in S$ and decreasing in $r\in(0,\infty)$.

It is shown in \cite{J} that this class is characterized by the stochastic
integral representation (\ref{ir}) where $\mu \in I(\mathbb{R})$ and
$h(t)=1_{\left[  0,1\right]  }(t)t$.

Let $\mu_{h}\in U(\mathbb{R})$ with such a representation. The corresponding
random matrix models (\ref{rmm}) for the free Jurek distributions $\Lambda
(\mu_{h})$ are of the form%
\[
\left(  M_{h}^{d}=\int_{0}^{1}td\Psi_{t}^{d}\right)  _{d\geq1}\text{.}%
\]

Let $(\lambda,\nu_{\xi})$ and $(\lambda_{h},\nu_{h_{\xi}})$ be the polar
decompositions of the L\'{e}vy measures of $\mu$ and $\mu_{h}$, respectively.
It is proved in \cite{AM} that the corresponding $l$-function in
(\ref{lfunction}) for the radial component $\nu_{h_{\xi}}$ is expressed as
$l_{h_{\xi}}(r)=\int_{r}^{\infty}x^{-1}\nu_{\xi}(dx)$. Therefore the L\'{e}vy
measures of these random matrix models are expressed as%
\[
\nu_{M_{h}^{d}}(B)=d\int_{\mathbb{\tilde{S}}_{\mathbb{M}_{d}}^{+}}\omega
_{d}\left(  dV\right)  \int_{S}\lambda_{h}\left(  d\xi \right)  \int
_{0}^{\infty}dr\left(  \int_{r}^{\infty}x^{-1}\nu_{\xi}(dx)\right)  1_{B}(r\xi
V)\text{.}%
\]

\textbf{Example 9.} The class of $A$ distributions $A(\mathbb{R})$ introduced
in \cite{MPAS} is bigger than the Jurek class $U(\mathbb{R})$. It is the class
of infinitely divisible distributions on $\mathbb{R}$ with L\'{e}vy measures
$\nu$ of the form%
\[
\nu(B)=\int_{\mathbb{R}\text{%
$\backslash$%
}\left \{  0\right \}  }\rho \left(  dx\right)  \int_{0}^{\infty}a_{1}\left(
r;\left \vert x\right \vert \right)  1_{B}\left(  r\frac{x}{\left \vert
x\right \vert }\right)  dr\qquad B\in \mathcal{B}\left(  \mathbb{R}\right)
\text{,}%
\]
where $\rho$ is a L\'{e}vy measure on $\mathbb{R}$; and $a_{1}\left(
r;s\right)  $ is the one-sided arcsine density with parameter $s>0$, that is,%
\[
a_{1}\left(  r;s\right)  =\left \{
\begin{array}
[c]{cc}%
2\pi^{-1}\left(  s-r^{2}\right)  ^{-1/2} & 0<r<s^{1/2}\\
0 & \text{otherwise.}%
\end{array}
\right.
\]
It is shown in \cite{MPAS} that the class $A(\mathbb{R})$ is characterized by
the stochastic integral representation (\ref{ir}) where $\mu \in I(\mathbb{R})$
and $h(t)=1_{\left[  0,1\right]  }(t)\cos \left(  \frac{\pi}{2}t\right)  $.

Let $\mu_{h}\in A(\mathbb{R})$ with such a representation. In this case the
random matrix models (\ref{rmm}) for the corresponding free $A$ distributions
$\Lambda(\mu_{h})$ are%
\[
\left(  M_{h}^{d}=\int_{0}^{1}\cos \left(  \frac{\pi}{2}t\right)  d\Psi_{t}%
^{d}\right)  _{d\geq1}%
\]
with L\'{e}vy measures, see (\ref{Levymasurermm}),
\[
\nu_{M_{h}^{d}}\left(  B\right)  =d\int_{\mathbb{\tilde{S}}_{\mathbb{M}_{d}%
}^{+}}\omega_{d}\left(  dV\right)  \int_{\mathbb{R}}\nu_{\mu}\left(
dx\right)  \int_{0}^{1}1_{B}\left(  xV\cos \left(  \frac{\pi}{2}t\right)
\right)  dt\text{.}%
\]

\textbf{Example 10.} The class of Type G distributions $G_{sym}(\mathbb{R})$
is also characterized (see Example $6$) by the stochastic integral
representation (\ref{ir}) where $\mu \in A\left(  \mathbb{R}\right)  $ and
$h(t)=1_{\left(  0,1/2\right)  }(t)\left(  \log \frac{1}{t}\right)  ^{1/2}$,
see \cite{MPAS}.\newline Let $\mu_{h}\in G_{sym}(\mathbb{R})$ has this
alternative representation, that is%
\begin{equation}
\mu_{h}=\mathcal{L}\left(  \int_{0}^{1/2}\left(  \log \frac{1}{t}\right)
^{1/2}dX_{t}^{\left(  \mu \right)  }\right)  \quad \mu \in A\left(
\mathbb{R}\right)  \text{.} \label{sir typeg}%
\end{equation}
The corresponding random matrix models (\ref{rmm}) for the free Type G
distributions $\Lambda(\mu_{h})$ are given by%
\[
\left(  M_{h}^{d}=\int_{0}^{1/2}\left(  \log \frac{1}{t}\right)  ^{1/2}%
d\Psi_{t}^{d}\right)  _{d\geq1}%
\]
with L\'{e}vy measures%
\[
\nu_{M_{h}^{d}}\left(  B\right)  =d\int_{\mathbb{\tilde{S}}_{\mathbb{M}_{d}%
}^{+}}\omega_{d}\left(  dV\right)  \int_{\mathbb{R}}\nu_{\mu}\left(
dx\right)  \int_{0}^{1/2}1_{B}\left(  xV\left(  \log \frac{1}{t}\right)
^{1/2}\right)  dt\text{.}%
\]

\textbf{Example 11}. The class of generalized Type $G$ distributions
$G(\mathbb{R})$ introduced in \cite{MS} is the class of infinitely divisible
distributions on $\mathbb{R}$ whose L\'{e}vy measures have polar decomposition
$(\lambda,\nu_{\xi})$ such that
\[
\nu_{\xi}(dr)=g_{\xi}(r^{2})dr\text{,}%
\]
where $g_{\xi}(r)$ is measurable in $\xi \in S$ and completely monotone in
$r\in(0,\infty)$. Distributions of this class are not necessarily symmetric.
If $\mu \in G(\mathbb{R})$ is symmetric then $\mu \in G_{sym}(\mathbb{R})$, that
is $\mu$ is of Type $G$.

It is proved in \cite{MPAS} that the class of generalized type $G$
distributions is characterized by the stochastic integral representation
(\ref{ir}) where $\mu \in A\left(  \mathbb{R}\right)  $ and $h(t)=1_{\left(
0,1\right)  }(t)\left(  \log \frac{1}{t}\right)  ^{1/2}$.

Let $\mu_{h}\in G(\mathbb{R})$ has such a representation, that is%
\begin{equation}
\mu_{h}=\mathcal{L}\left(  \int_{0}^{1}\left(  \log \frac{1}{t}\right)
^{1/2}dX_{t}^{\left(  \mu \right)  }\right)  \quad \mu \in A\left(
\mathbb{R}\right)  \text{.} \label{sir genertypeg}%
\end{equation}
This representation (\ref{sir genertypeg}) is not necessarily a symmetric
generalization of (\ref{sir typeg}).

In this case the corresponding random matrix models (\ref{rmm}) for the free
generalized type $G$ distributions $\Lambda(\mu_{h})$ are%
\[
\left(  M_{h}^{d}=\int_{0}^{1}\left(  \log \frac{1}{t}\right)  ^{1/2}d\Psi
_{t}^{d}\right)  _{d\geq1}\text{.}%
\]

\end{document}